\documentclass{article}

\usepackage{arxiv}

\usepackage[T1]{fontenc}    
\usepackage{hyperref}       
\usepackage{url}            
\usepackage{CJK}
\usepackage{booktabs}       
\usepackage{amsfonts}       
\usepackage{nicefrac}       
\usepackage{microtype}      
\usepackage{lipsum}
\usepackage{graphicx}
\usepackage{amsmath}
\usepackage[all]{xy}
\usepackage{amsthm}
\usepackage{eucal}
\usepackage{framed}

\newtheorem{theorem}{Theorem}
\theoremstyle{definition}

\newtheorem{example}[theorem]{Example}

\newtheorem{lemma}{Lemma}
\newtheorem{remark}{Remark}
\newtheorem{corollary}{Corollary}
\newtheorem{proposition}{Proposition}

\title{A New Series Representation Involving Root Of Unity For The Values Of Riemann Zeta Function At Integer Arguments }

\author{
  Xiaowei Wang(Potsdam)\thanks{This paper is written in July 2020}
   \\
}

\begin{document}
\maketitle

\begin{abstract}
In this paper we provide a new series representation for the values of Riemann zeta function at integer arguments, namely: $      \zeta(m)=\sum_{n=1}^{\infty}\frac{m(-1)^{n-1}\Gamma(1-\omega_{m}n)...\Gamma(1-\omega_{m}^{m-1}n)}{n!n^m}$, where $n$ is an integer that lager than $1$ and $\omega$ is the $m$-th root of unity. This series converges quite fast. It's derived by some technique of infinite partial fraction decomposition. With this technique we also establish other useful formulas related to gamma function.
\end{abstract}


\section{Introduction}
It's well known that there are still a bunch of unsolved problems about the values of Riemann zeta function at integer arguments. For instance, except $\zeta(3)$ (It was proved that $\zeta(3)$ is irrational, but its transcendence remains unknown, see\cite{van}), the irrationality and transcendence of $\zeta(2m+1)$ are unknown. Whether $\frac{\zeta(2m+1)}{\pi^{2m+1}}$ are irrational? Is there any explicit relation between $\zeta(m_{1})$ and $\zeta(m_{2})$ for different positive integer $m_{1},m_{2}$?\\

In order to study all those problems, seeking for a new integral or series representation for $\zeta(m)$ is good direction to understand the number-theoretic properties of these constants. A typical successful example is that, F. Beukers \cite{beukers} used a special integral representation to prove the irrationality of $\zeta(3)$ in 1978.\\

In this paper we provide a new series representation of $\zeta(m)$ for all $m\in \mathbb{Z}_{\geq 2}$ by using some techniques of infinite partial fraction decomposition, namely\\
\begin{equation*}
  \zeta(m)=\sum_{n=1}^{\infty}\frac{m(-1)^{n-1}\Gamma(1-\omega_{m}n)...\Gamma(1-\omega_{m}^{m-1}n)}{n!n^m}
\end{equation*}
where $\omega_{m}=e^{2\pi i/m}$ is the $m$-th root of unity. One can prove that this series converges more faster than $\zeta(m)=\sum_{n=1}^{\infty}\frac{1}{n^m}$, but the disadvantage is that all terms are irrational.\\

\section{Main Result}
\begin{lemma}\label{lemma1}(Homogeneous partial fraction decomposition)\\
Let $a_{1},...,a_{n}$ be distinct complex number, $x\in \mathbb{C}\backslash \{a_{1},...,a_{n}\} $, then there exist $\mu_{1},...,\mu_{n}\in \mathbb{C}$ such that following identity is true,\\
\begin{equation}\label{pfd1}
  \prod_{i=1}^{n}\frac{1}{x-a_{i}}=\sum_{i=1}^{n}\frac{\mu_{i}}{x-a_{i}}
\end{equation}
where $\mu_{i}$ has explicit expression as following. They only depend on $a_{1},...,a_{n}$.
\begin{equation*}
  \mu_{i}=\prod_{j=1,j\neq i}^{n}\frac{1}{a_{i}-a_{j}}
\end{equation*}
Further, we have
\begin{equation*}
  \sum_{i=1}^{n}\mu_{i}=0
\end{equation*}
\end{lemma}

\begin{proof}
In order to show (\ref{pfd1}), we multiply $\prod_{i=1}^{n} (x-a_{i})$ on both side of (\ref{pfd1}). It becomes
\begin{equation*}
 \sum_{i=1}^{n}\mu_{i}\prod_{j=1,j\neq i}^{n}(x-a_{j})=1
\end{equation*}
Now let
\begin{equation*}
  p(x)=\sum_{i=1}^{n}\mu_{i}\prod_{j=1,j\neq i}^{n}(x-a_{j})=\sum_{i=1}^{n}\prod_{j=1,j\neq i}^{n}\frac{x-a_{j}}{a_{i}-a_{j}}
\end{equation*}
It's easy to see that $p(x)$ is a polynomial with degree $n-1$ and satisfying that $p(a_{i})=1$ for all $i=1,...,n$.  On the one hand we already found $n$ zeros of $p(x)-1$, on the other hand by the fundamental theorem of algebra, $p(x)-1$ has $n-1$ zeros. Therefore the only possible case is $p(x)-1\equiv 0$. That is
\begin{equation*}
\sum_{i=1}^{n}\mu_{i}\prod_{j=1,j\neq i}^{n}(x-a_{j})\equiv1
\end{equation*}
This means (\ref{lemma1}) is true. Finally, by comparing the coefficient of $x^{n-1}$ on both sides, we obtain
\begin{equation*}
  \sum_{i=1}^{n}\mu_{i}=0
\end{equation*}
\end{proof}

\begin{lemma}
For $m=2,3,...$, define $\Phi_{m}(z)=\prod_{n=1}^{\infty}\frac{n^m}{n^m-z^m}    $, then
\begin{equation*}
    \Phi_{m}(z)=\prod_{j=0}^{m-1}\Gamma(1-\omega_{m}^{j}z)=\exp{(\sum_{k=1}^{\infty}\frac{\zeta(mk)}{k}z^{mk})}
\end{equation*}
where the second identity holds only for $|z|<1$.
\end{lemma}
\begin{proof}
Firstly, recall the infinite product expression of Gamma function
\begin{equation*}
    \Gamma(z)=z^{-1}e^{-\gamma z}\prod_{n=1}^{\infty}\frac{ne^{z/n}}{z+n}, z\in \mathbb{C}\backslash \{0,-1,-2,...\}
\end{equation*}
Let $\omega_{m}=e^{2i\pi/m}$, it's easy to see that for $j=0,1,...,m-1$ and $z\neq \omega_{m}^{-j}, 2\omega_{m}^{-j}, 3\omega_{m}^{-j},...$
\begin{equation*}
    \Gamma(1-\omega_{m}^{j}z)=e^{\gamma \omega_{m}^{j}z}\prod_{n=1}^{\infty}\frac{ne^{-\omega_{m}^{j}z/n}}{n-\omega_{m}^{j}z}
\end{equation*}
Multiplying them together, we have
\begin{align*}
    \prod_{j=0}^{m-1}\Gamma(1-\omega_{m}^{j}z)&=\exp{(\gamma z \sum \omega_{m}^{j})}\prod_{n=1}^{\infty}\prod_{j=0}^{m-1}\frac{ne^{-z/n\sum\omega_{m}^{j}}}{n-\omega_{m}^{j}z}\\
    &=\prod_{n=1}^{\infty}\prod_{j=0}^{m-1}\frac{n}{n-\omega_{m}^{j}z}\\
    &=\prod_{n=1}^{\infty}\frac{n^m}{n^m-z^m}
\end{align*}
The second required identity is in fact due to the Taylor expansion of $\log{(\Gamma)}$ around $1$. Since
\begin{equation*}
    \log{\Gamma(1-z)}=\gamma z+\sum_{k=2}^{\infty}\frac{\zeta(k)}{k}z^{k}, |z|<1
\end{equation*}
then for $j=0,1,...,m-1$
\begin{equation*}
    \Gamma(1-\omega_{m}^{j}z)=\exp{(\gamma \omega_{m}^{j}z+\sum_{k=2}^{\infty}\frac{\zeta(k)}{k}(\omega_{m}^{j}z)^{k})}, |z|<1
\end{equation*}
Multiplying them together, we obtain
\begin{equation*}
    \prod_{j=0}^{m-1}\Gamma(1-\omega_{m}^{j}z)=\exp{(\sum_{k=1}^{\infty}\frac{\zeta(mk)}{k}z^{mk})},|z|<1
\end{equation*}
\end{proof}

\begin{theorem}
For $m\in \mathbb{Z}_{\geq 2}$, the value of Riemann zeta function at integer  argument $\zeta(m)$ has following series representation\\
\begin{equation}\label{generalzeta}
  \zeta(m)=\sum_{n=1}^{\infty}\frac{m(-1)^{n-1}\Gamma(1-\omega_{m}n)...\Gamma(1-\omega_{m}^{m-1}n)}{n!n^m}
\end{equation}
where $\omega_{m}=e^{\frac{2i\pi}{m}}$ is the $m$-th root of unity.
\end{theorem}
\begin{proof}
For fixed $m$, let
\begin{equation*}
    f_{N}(z)=\frac{1}{z}\prod_{n=1}^{N}\frac{n^m}{n^m-z^m}
\end{equation*}
then $\lim_{N\rightarrow \infty}f_{N}(z)=\frac{1}{z}\Phi_{m}(z)$. Now use Lemma 1 to find the coefficients of partial fraction decomposition of $f_{N}(z)$. Rewrite
\begin{equation*}
    f_{N}(z)=\frac{1}{z}\prod_{n=1}^{N}\prod_{j=0}^{m-1}\frac{-\omega_{m}^{j}n}{z-\omega_{m}^{j}}
\end{equation*}
By Lemma 1, we have\\
\begin{equation}\label{fn}
     f_{N}(z)=\frac{\lambda_{0}(N)}{z}+\sum_{n=1}^{N}\frac{\lambda_{n,j}(N)}{z-\omega_{m}^{j}}
\end{equation}
where
\begin{align*}
   & \lambda_{0}(N)\equiv 1\\
   &\lambda_{n,j}(N)=(-1)^{N}(N!)^{m}(n\omega_{m}^{j})^{-1}\prod_{(s,t)\in \Lambda_{n,j}^{N}}(-s\omega_{m}^{t}+n\omega_{m}^{j})^{-1}
\end{align*}
with
\begin{equation*}
 \Lambda_{n,j}^{N}:=\{(s,t):s=1,2,...,N;t=0,1,...,m-1;(s,t)\neq (n,j)\}   
\end{equation*}
Note that $\lambda_{n,j}$ are actually independent to $j$ for all $n=1,2,...$, that is\\
\begin{equation}\label{lambdaeq}
    \lambda_{n,0}=\lambda_{n,1}=...=\lambda_{n,m-1}
\end{equation}
The reason is following, for $j\neq 0$
\begin{align*}
    (n\omega_{m}^{j})^{-1}\prod_{(s,t)\in \Lambda_{n,j}^{N}}(-s\omega_{m}^{t}+n\omega_{m}^{j})^{-1}&=n^{-1}(\omega_{m}^{j})^{mj-1}\prod_{(s,t)\in \Lambda_{n,j}^{N}}(-s\omega_{m}^{t}+n\omega_{m}^{j})^{-1}\\
    &=n^{-1}\prod_{(s,t)\in \Lambda_{n,j}^{N}}(-s\omega_{m}^{t-j}+n)^{-1}\\
    &=n^{-1}\prod_{(s,t')\in \Lambda_{n,0}^{N}}(-s\omega_{m}^{t'}+n)^{-1}\\
\end{align*}
This manifests that (\ref{lambdaeq}) is true. Hence it remains to find $\lambda_{n,0}(N)$. For convenience we omit the index of $j$, directly denote it by $\lambda_{n}(N)$. By some simple computation,
\begin{align*}
    \lambda_{n}(N)&=\frac{-n^m}{n(n\omega_{m}-n)(n\omega_{m}^2-n)...(n\omega_{m}^{m-1}-n)}\prod_{s=1,s\neq n}^{N}\frac{s^m}{s^m-n^m}\\
    &=-\frac{1}{m}\prod_{s=1,s\neq n}^{N}\frac{s^m}{s^m-n^m}\\
\end{align*}
Therefore\\
\begin{equation}\label{lambdainf}
    \lambda_{n}:=\lim_{N\rightarrow \infty}\lambda_{n}(N)=-\frac{1}{m}\prod_{s=1,s\neq n}^{\infty}\frac{s^m}{s^m-n^m}
\end{equation}
In fact, there is another way to formulate $\lambda_{n}$.\\
\begin{align*}
    \lambda_{n}&=\lim_{z\rightarrow n}-\frac{n^m-z^m}{m n^m}\prod_{s=1}^{\infty}\frac{s^m}{s^m-z^m}\\
    &=\lim_{z\rightarrow n}-\frac{n^m-z^m}{m n^m}\Gamma(1-z)\Gamma(1-\omega_{m}z)...\Gamma(1-\omega_{m}^{m-1}z)\\
    &=\frac{(-1)^{n}}{n!}\Gamma(1-\omega_{m}n)...\Gamma(1-\omega_{m}^{m-1}n)
\end{align*}
plug into (\ref{fn}) we obtain\\
\begin{equation}\label{phi}
    \Phi_{m}(z)=1+\sum_{n=1}^{\infty}\frac{m(-1)^n \Gamma(1-\omega_{m}n)...\Gamma(1-\omega_{m}^{m-1}n)}{n!}\frac{z^m}{z^m-n^m}
\end{equation}
In order to show the normal convergence of $\Phi_{m}(z)$ in $\mathbb{C}$, consider\\
\begin{equation*}
      |\lambda_{n}|=\frac{1}{m}\prod_{s=1,s\neq n}^{\infty}|\frac{s^m}{s^m-n^m}|<\frac{1}{2}\prod_{s=1,s\neq n}^{\infty}|\frac{s^2}{s^2-n^2}|=1
\end{equation*}
The last equality is due to\\
\begin{equation*}
 -\frac{1}{2}\prod_{s=1,s\neq n}^{\infty}\frac{s^2}{s^2-n^2}=\frac{(-1)^n\Gamma(1+n) }{n!}=(-1)^n  
\end{equation*}
Therefore one can estimate (\ref{phi}) by\\
\begin{align*}
    &\sum_{n=1}^{\infty}|\frac{m(-1)^n \Gamma(1-\omega_{m}n)...\Gamma(1-\omega_{m}^{m-1}n)}{n!}\frac{z^m}{z^m-n^m}|\\< & m\sum_{n=1}^{\infty}|\frac{z^m}{n!(z^m-n^m)}|\\
    < & m M^m\sum_{n=1}^{\infty}\frac{1}{n!(n^m-M^{m})}\\
\end{align*}
for $|z|<M$. Hence $\Phi_{m}(z)$ is normally converges to a meromorphic function in $\mathbb{C}$. The poles are $n\omega_{m}^{j}$, $n=1,2,...$, $j=0,1,...,m-1$. By the normal convergence $\Phi_{m}(z)$ can be rearranged as\\
\begin{equation*}
 \Phi_{m}(z)+1+\sum_{k=1}^{\infty}\sum_{n=1}^{\infty}\frac{m(-1)^{n-1} \Gamma(1-\omega_{m}n)...\Gamma(1-\omega_{m}^{m-1}n)}{n!n^{km}} z^{km} ,|z|<1 
\end{equation*}
On the other hand, it follows from Lemma 2 that\\
\begin{equation*}
   \Phi_{m}(z)= \exp{(\sum_{k=1}^{\infty}\frac{\zeta(mk)}{k}z^{mk})}=1+\zeta(m)z^{km}+...,\text{  }|z|<1
\end{equation*}
Finally, compare the coefficients we obtain\\
\begin{equation*}
      \zeta(m)=\sum_{n=1}^{\infty}\frac{m(-1)^{n-1}\Gamma(1-\omega_{m}n)...\Gamma(1-\omega_{m}^{m-1}n)}{n!n^m}
\end{equation*}
\end{proof}

\begin{theorem}
Suppose that $r\in \mathbb{Z}^{+}$, then\\
\begin{equation}\label{powerseries}
    \frac{1}{r!}\frac{d^r}{dz^r}|_{z=0}\exp{(\sum_{k=1}^{\infty}\frac{\zeta(mk)}{k}z^{mk})}=\begin{cases}
    1, \text{ if } r=0\\
    0, \text{ if } r\neq 0 \text{ and }\mod{(r,m)}\neq 0\\
    \sum_{n=1}^{\infty}\frac{m(-1)^{n-1} \Gamma(1-\omega_{m}n)...\Gamma(1-\omega_{m}^{m-1}n)}{n!n^{rm}} , \text{ if } r\neq 0 \text{ and} \mod{(r,m)}= 0
    \end{cases}
\end{equation}
\end{theorem}
\begin{proof}
In the proof of last theorem it's shown that\\
\begin{equation*}
 \Phi_{m}(z)+1+\sum_{k=1}^{\infty}\sum_{n=1}^{\infty}\frac{m(-1)^{n-1} \Gamma(1-\omega_{m}n)...\Gamma(1-\omega_{m}^{m-1}n)}{n!n^{km}} z^{km} ,|z|<1 
\end{equation*}
Recall\\
\begin{equation*}
    \Phi_{m}(z)=\exp{(\sum_{k=1}^{\infty}\frac{\zeta(mk)}{k}z^{mk})},|z|<1 
\end{equation*}
by comparing the coefficients, one can see the result.\\
\end{proof}

\section{Some Estimates}

\begin{lemma}\label{compare}(Comparing of the coefficients)
Let $\rho_{m}(n):=\prod_{s=1,s\neq n}^{\infty}|\frac{s^m}{s^m-n^m}|=\frac{m\Gamma(1-\omega_{m}n)...\Gamma(1-\omega_{m}^{m-1}n)}{n!}$, then for any fixed $n\in \mathbb{Z}^{+}$, $\rho_{m}(n)$ is decreasing with respect to $m$, namely
\begin{equation*}
    \rho_{m+1}(n)<\rho_{m}(n)
\end{equation*}
\end{lemma}
\begin{proof}
Observe that $\frac{s^m}{s^m-n^m}=(1-(\frac{n}{s})^m)^{-1}$.\\
Case 1. For fixed $n,s$ such that $n>s$, $|(1-(\frac{n}{s})^m)^{-1}|=((\frac{n}{s})^m-1)^{-1}$ is decreasing with respect to $m$.\\
Case 2. For fixed $n,s$ such that $n<s$, $|(1-(\frac{n}{s})^m)^{-1}|=(1-(\frac{n}{s})^m)^{-1}$ is also decreasing with respect to $m$.\\
Therefore $\rho_{m}(n)=\prod_{s=1,s\neq n}^{\infty}|\frac{s^m}{s^m-n^m}|$ is decreasing with respect to $m$.\\
\end{proof}

\begin{theorem}
Assume that $2m+1$ is an odd integer lager than $3$, for the series representation
\begin{equation*}
    \zeta(2m+1)=\sum_{n=1}^{\infty}\frac{(2m+1)(-1)^{n-1}\Gamma(1-\omega_{2m+1}n)...\Gamma(1-\omega_{2m+1}^{2m}n)}{n!n^{2m+1}}=\sum_{n=1}^{\infty}\frac{(-1)^{n-1}\rho_{2m+1}(n)}{n^{2m+1}}
\end{equation*}
we have an estimate for the absolute value of each term as following. \begin{equation*}
   \frac{2(m+1)\pi^m}{e^{\pi (m+1)n}n^{m+1}}< |\frac{\rho_{2m+1}(n)}{n^{2m+1}}|<\frac{2m(4\pi)^{m-1}}{e^{\sqrt{2}mn}n^{m+2}}
\end{equation*}
where the second inequality holds for $n>0.035m$

\end{theorem}
\begin{proof}
By lemma \ref{compare} we have $\rho_{2m+2}(n)<\rho_{2m+1}(n)<\rho_{2m}(n)$. It remains to find upper bound and lower bound of $\rho_{2r}(n)$. Obviously $\rho_{2}(n)=2$ for all $n$, therefore in following we can always assuming that $r\geq 2$. Firstly realize that\\
\begin{align*}
    \Gamma(1-\omega_{2r}n)...\Gamma(1-\omega_{2r}^{2r-1}n)&=\Gamma(1+n)\prod_{j=1}^{r-1}\Gamma(1-\omega_{2r}^{j}n)\Gamma(1+\omega_{2r}^{j}n)\\
    &=n!\prod_{j=1}^{r-1}\pi\omega_{2r}^{j}n\csc(\pi\omega_{2r}^{j}n)\\
    &=n!(\pi n)^{r-1} e^{i\pi(r-1)/2} \prod_{j=1}^{r-1}\csc(\pi\omega_{2r}^{j}n)
\end{align*}
Therefore $\rho_{2r}(n)=2r(\pi n)^{r-1}\prod_{j=1}^{r-1}|\csc(\pi\omega_{2r}^{j}n)|$. It remains to estimate the product of $|\csc(.)|$. Note that for $a,b\in \mathbb{R}, b>0,a\neq 0$, $|\sin(a+bi)|=\frac{1}{2}|e^{b-ai}-e^{-b+ai}|\in (\frac{e^b-e^{-b}}{2},\frac{e^b+e^{-b}}{2})$. Hence\\
\begin{equation*}
    \frac{1}{4}e^{|\Im{(z)}|}<|\sin(z)|< e^{|\Im{(z)}|}
\end{equation*}
The first inequality holds for $|\Im{(z)}|>\frac{\log 2}{2}$, the second inequality holds for $z\in \mathbb{C}$. Therefore one can implies that\\
\begin{equation*}
    \exp(-n\pi\sum_{j=1}^{r-1}\Im{(\omega_{2r}^{j})})<\prod_{j=1}^{r-1}|\csc(\pi\omega_{2r}^{j}n)|< 4^{r-1} \exp(-n\pi\sum_{j=1}^{r-1}\Im{(\omega_{2r}^{j})})
\end{equation*}
which the second inequality holds for $n>0.035r$ (since $\frac{\log 2}{2\pi }\csc(\pi/r)>0.035 r$). On the other hand, by the formula\\
\begin{equation*}
    \sum_{j=1}^{r-1}\sin(\frac{\pi j}{r})=\frac{\sin(\frac{\pi(r-1)}{2r})}{\sin(\frac{\pi}{2r})}
\end{equation*}
It's easy to check the fact
\begin{equation*}
    \frac{\sqrt{2}}{\pi}r<\frac{\sin(\frac{\pi(r-1)}{2r})}{\sin(\frac{\pi}{2r})}< r
\end{equation*}
Therefore\\
\begin{equation*}
   \frac{1}{e^{n\pi r}}< \prod_{j=1}^{r-1}|\csc(\pi\omega_{2r}^{j}n)|< \frac{4^{r-1}}{e^{n\sqrt{2} r}}
\end{equation*}
Now let $r=m$ and $m+1$ respectively, it turns out that
\begin{equation*}
    \rho_{2m}(n)< \frac{2m(4\pi n)^{m-1}}{e^{n\sqrt{2} m}}
\end{equation*}
holds for $n>0.035 m$, and\\
\begin{equation*}
    \rho_{2m+2}(n)> \frac{2(m+1)(\pi n)^{m}}{e^{n\pi (m+1)}}
\end{equation*}
holds for all $n$. Finally, we can conclude that\\
\begin{equation*}
   \frac{2(m+1)\pi^m}{e^{\pi (m+1)n}n^{m+1}}< |\frac{\rho_{2m+1}(n)}{n^{2m+1}}|<\frac{2m(4\pi)^{m-1}}{e^{\sqrt{2}mn}n^{m+2}}
\end{equation*}
where the second inequality holds for $n>0.035m$.\\
\end{proof}
In the appendix we show some numerical comparison for those two series representations of $\zeta(m)$.\\

\section{Some Special Examples}
The case of $m=3$ is a little bit different to the case of other odd number. To see this, observe the formula (\ref{generalzeta}),\\
\begin{equation}\label{zeta3}
    \zeta(3)=\sum_{n=1}^{\infty}\frac{3(-1)^{n-1}\Gamma(1-\omega_{3}n)\Gamma(1-\omega_{3}^2 n)}{n! n^3}
\end{equation}
Since $\omega_{3}=-\frac{1}{2}+\frac{\sqrt{3}}{2}i$, $\omega_{3}^2=\bar{\omega}_{3}=-\frac{1}{2}-\frac{\sqrt{3}}{2}i$. It makes that $1-\omega_{3}n$ and $1-\omega_{3}^2n$ are conjugate and have rational real part. Moreover, their real parts are integers or half integers. The case of odd $m>3$ is totally another story. Although $1-\omega_{m}^{j}n$ and $1-\omega_{m}^{m-j}n$ are conjugate, the real part is $1-\cos(\frac{2\pi}{m})$. It cannot be a rational number for odd $m$ that larger than $3$. It's still unknown if there is some closed-form, not obvious functional equation for gamma function that is applicable for a number with irrational real part.\\
According to the particularity of the case "$m=3$" we give some variant of (\ref{zeta3}) in following.\\
\\
\begin{proposition}(Sine and the root of unity)
\begin{equation}\label{zeta3com}
 \zeta(3)=\sum_{n=1}^{\infty}\frac{3\pi(-1)^{n-1} \omega_{3}^2 (1+\omega_{3}^2 n)(2+\omega_{3}^2 n)...(n+\omega_{3}^2 n)}{n!n^2\sin(\pi \omega_{3}^2 n)}
\end{equation}
\end{proposition}
\begin{proof}
To prove this, just need to note that $1+\omega_{3}^2 n+n=1-\omega_{3}n$, 
\begin{equation*}
  \Gamma(1-\omega_{3}^2 n)=\frac{\pi \omega_{3}^2n}{\Gamma(1+\omega_{3}^2 n)\sin(\pi \omega_{3}^2 n)}
\end{equation*}
Hence\\
\begin{align*}
  \Gamma(1-\omega_{3} n)\Gamma(1-\omega_{3}^2 n) & =\frac{\pi \omega_{3}^2n \Gamma(1-\omega_{3} n)}{\Gamma(1+\omega_{3}^2 n)\sin(\pi \omega_{3}^2 n)} \\
   & =\frac{\pi \omega_{3}^2n \Gamma(1+\omega_{3}^2 n+n)}{\Gamma(1+\omega_{3}^2 n)\sin(\pi \omega_{3}^2 n)} \\
   &=\frac{\pi \omega_{3}^2n(1+\omega_{3}^2 n)(2+\omega_{3}^2 n)...(n+\omega_{3}^2 n)}{\sin(\pi \omega_{3}^2 n)}
\end{align*}
By (\ref{zeta3}) we obtain what we need.\\
\end{proof}
\begin{proposition}(Sinh and Cosh)
\begin{equation}
    \zeta(3)=\sum_{d=1}^{\infty}\frac{3\pi P(d)}{8(2d-1)!(d-\frac{1}{2})^3\cosh{(\sqrt{3}\pi(d-\frac{1}{2}))}}-\sum_{d=1}^{\infty}\frac{3\sqrt{3}\pi Q(d)}{8(2d)!d^2\sinh{(\sqrt{3}\pi d)}}
\end{equation}
where\\
\begin{align*}
&P(d)=\prod_{k=1}^{d}((k-\frac{1}{2})^2+3(d-\frac{1}{2})^2)\\
&Q(d)=\prod_{k=1}^{d}(k^2+3d^2)
\end{align*}
\end{proposition}
\begin{proof}
Since\\
\begin{equation*}
\Gamma(1-\omega_{3}n)\Gamma(1-\omega_{3}^2n)=\Gamma(1+\frac{n}{2}+\frac{\sqrt{3}n}{2}i)\Gamma(1+\frac{n}{2}-\frac{\sqrt{3}n}{2}i)
\end{equation*}
for $n=2d$, it becomes\\
\begin{equation*}
    \Gamma(1+d+\sqrt{3}di)\Gamma(1+d-\sqrt{3}di)=\frac{\sqrt{3}\pi d Q(d)}{\sinh{(\sqrt{3}\pi d)}}
\end{equation*}\\
On the other hand, for $n=2d-1$, it becomes\\
\begin{equation*}
     \Gamma(\frac{1}{2}+d+\frac{\sqrt{3}}{2}(2d-1)i)\Gamma(\frac{1}{2}+d-\frac{\sqrt{3}}{2}(2d-1)i)=\frac{\pi P(d)}{\cosh{(\sqrt{3}\pi(d-\frac{1}{2})})}
\end{equation*}
plug into (\ref{zeta3}) we get what required.
\end{proof}

\begin{example}
According to (\ref{powerseries}), one has\\
\begin{equation*}
\zeta(3)^2+\zeta(6)=\sum_{n=1}^{\infty}\frac{6(-1)^{n-1} \Gamma(1-\omega_{3}n)\Gamma(1-\omega_{3}^{2}n)}{n!n^{6}}
\end{equation*}
\end{example}

\begin{example}
For $m=4$, the representation formula (\ref{generalzeta}) becomes\\
\begin{equation*}
    \zeta(4)=\sum_{n=1}^{\infty}\frac{4(-1)^{n-1}\Gamma(1-in)\Gamma(1+in)}{n^4}=4\pi\sum_{n=1}^{\infty}\frac{(-1)^{n-1} }{n^3\sinh(\pi n ) }
\end{equation*}
By some short computation one has\\
\begin{equation*}
\frac{\pi^3}{360}=\sum_{n=1}^{\infty}\frac{(-1)^{n-1}}{n^3\sinh(\pi n)}
\end{equation*}
In fact, there is a similar identity derived by Ramanujan \cite{rama}\\
\begin{equation*}
\frac{7\pi^3}{180}=\sum_{n=1}^{\infty}\frac{\coth(\pi n)}{n^3}
\end{equation*}
It seems that there are some connection between them, but whether there is an elementary arithmetic relation is yet unknown.\\

\end{example}

\section{The product of Gamma Function}
This section is relative independent to the preceding results, here we discuss the infinite partial fraction decomposition of some product of Gamma function. At first, we establish following useful formula.
\begin{lemma}\label{lemma3}(Partial Fraction Summation)\\
Let $F(z)=\prod_{n=1}^{\infty}(1-(\frac{z}{a_{n}})^2)$ where series $(a_{n})_{n}$ is an complex sequence satisfying that: \\
1, $a_{n}\neq 0$ for all $n$,\\
2, $a_{i}\neq a_{j}$ for all $i\neq j$,\\
3, $\sum \frac{1}{|a_{n}|^2}$ converges,\\
then the following identity is true for all positive integer $N$\\
\begin{equation*}
\sum_{n=1}^{N}\frac{1}{a_{n}^2}=\sum_{n=1}^{N} \prod_{s=1,s\neq n}^{N}\frac{a_{s}^2}{a_{s}^2-a_{n}^2}\frac{1}{a_{n}^2}
\end{equation*}
Moreover, the identity also holds for $N=\infty$, namely\\
\begin{equation*}
\sum_{n=1}^{\infty}\frac{1}{a_{n}^2}=\sum_{n=1}^{\infty} \frac{-2\lambda_{n}}{a_{n}^2}
\end{equation*}
where\\
\begin{equation*}
 \lambda_{n}=-\frac{1}{2}\prod_{s=1,s\neq n}^{\infty}\frac{a_{s}^2}{a_{s}^2-a_{n}^2}=\frac{1}{a_{n}F'(a_{n})}
\end{equation*}
\end{lemma}
\begin{proof}
For finite $N$ the identity is just a simple corollary of lemma \ref{lemma1}. It can be easily proved by consider the partial fraction decomposition of $\prod_{n=1}^{N}(1-(\frac{z}{a_{n}})^2)^{-1}$.\\
In follow we discuss the case $N\rightarrow \infty$. Let $G(z)=(zF(z))^{-1}$, then\\
\begin{equation*}
  G(z)=\frac{1}{z}\prod_{n=1}^{\infty}\frac{a_{n}^2}{a_{n}^2-z^2}=\frac{1}{z}\exp(-\sum_{n=1}^{\infty}\log(1-(\frac{z}{a_{n}})^2))
\end{equation*}
Due to the convergence of $\sum \frac{1}{|a_{n}|^2}$, one has the following Taylor expansion of $-\log(1-(\frac{z}{a_{n}})^2)$ around $z=0$ in $|z|<\inf{|a_{n}|}$.\\
\begin{equation*}
-\log(1-(\frac{z}{a_{n}})^2)=\sum_{k=1}^{\infty}\sum_{n=1}^{\infty}\frac{1}{ka_{n}^{2k}}z^{2k}\\
\end{equation*}
This makes\\
\begin{equation}\label{ek}
  G(z)=\frac{1}{z}\sum_{k=0}^{\infty}E_{2k}z^{2k}
\end{equation}
where\\
\begin{equation*}
   E_{r}=\frac{1}{r!}\frac{d^r}{dz^r}|_{z=0} exp(-\sum_{n=1}^{\infty}\log(1-(\frac{z}{a_{n}})^2))
\end{equation*}
i.e.\\
\begin{align*}
   & E_{0}=1 \\
   & E_{2}=\sum_{n=1}^{\infty}\frac{1}{a_{n}^2}\\
   & E_{4}=\frac{1}{2}(\sum_{n=1}^{\infty}\frac{1}{a_{n}^2})^2+\frac{1}{2}\sum_{n=1}^{\infty}\frac{1}{(a_{n})^4}\\
   ...
\end{align*}
On the other hand, Let $N\in \mathbb{Z}^{+}$, consider the approximation\\
\begin{equation*}
    G_{N}(z)\rightarrow G(z),(N\rightarrow \infty)
\end{equation*}
where\\
\begin{equation*}
     G_{N}(z)=\frac{\lambda_{0}(N)}{z}+\sum_{n=1}^{N}\frac{\lambda_{+n}(N)}{z-a_{n}}+\frac{\lambda_{-n}(N)}{z+a_{n}}
\end{equation*}
Note that $G_{N}(z)$ is odd function, therefore $\lambda_{+n}(N)=\lambda_{-n}(N)$. Let them denoted by $\lambda_{n}(N)$. By lemma \ref{lemma1} we have\\
\begin{equation*}
\lambda_{n}(N)=-\frac{1}{2}\prod_{s=1,s\neq n}^{N}\frac{a_{s}^2}{a_{s}^2-a_{n}^2}
\end{equation*}
and $\lambda_{0}(N)=1$ for all $N$. $\lambda_{n}(N)$ can be rewritten as\\

 If $\lambda_{n}:=\lim_{N \rightarrow \infty}\lambda_{n}(N)$ exist, then\\
\begin{align*}
  G(z)&=\frac{1}{z}+\sum_{n=1}^{\infty}\lambda_{n}(\frac{1}{z+a_{n}}+\frac{1}{z-a_{n}})\\
\end{align*}
 It still follows from the convergence of $\sum \frac{1}{|a_{n}|^2}$, we have
\begin{equation}\label{gexp2}
  G(z)=\frac{1}{z}+\sum_{n=1}^{\infty}\lambda_{n}(\frac{1}{z+a_{n}}+\frac{1}{z-a_{n}})
  =\frac{1}{z}-\sum_{n=1}^{\infty} \frac{2\lambda_{n}}{a_{n}^2}z-\sum_{n=1}^{\infty} \frac{2\lambda_{n}}{a_{n}^4}z^3-\sum_{n=1}^{\infty} \frac{2\lambda_{n}}{a_{n}^6}z^5-...\\
\end{equation}
As for showing that $\lambda_{n}=\frac{1}{a_{n}F'(a_{n})}$, just need to note\\
\begin{equation*}
    \lambda_{n}=-\frac{1}{2}\prod_{s=1,s\neq n}^{\infty}\frac{a_{s}^2}{a_{s}^2-a_{n}^2}=-\frac{1}{2}\lim_{z-a_{n}\rightarrow 0}\frac{a_{n}^2-z^2}{a_{n}^2}\frac{1}{F(z)}=\frac{1}{a_{n}F'(a_{n})}
\end{equation*}

Finally, comparing the coefficients of (\ref{ek}) and (\ref{gexp2}), we obtain\\
\begin{equation*}
\sum_{n=1}^{\infty}\frac{1}{a_{n}^2}=\sum_{n=1}^{\infty} \frac{-2}{F'(a_{n})}\frac{1}{a_{n}^3}
\end{equation*}
\end{proof}

\begin{lemma}\label{lemma4}
The following identity is true.\\
\begin{equation*}
  \frac{\Gamma(a+z)\Gamma(a-z)}{\Gamma(a)^2}=\prod_{n=1}^{\infty}(1-(\frac{z}{a+n-1})^2)^{-1}
\end{equation*}
\end{lemma}
\begin{proof}
See \cite{htf}
\end{proof}

\begin{theorem}(Partial Fraction Decomposition Formula of Gamma Function)\\ Following identity holds for $\Re{(a)}\leq 1$, $|z|<|a|$.\\
\begin{equation}\label{dbgamma}
\Gamma(a+z)\Gamma(a-z)=\Gamma(a)^2+\sum_{k=0}^{\infty}\frac{(-1)^{k-1} \Gamma(2a+k) }{(a+k)k!}\frac{2z^2}{z^2-(a+k)^2}
\end{equation}
\end{theorem}
\begin{proof}
Suppose that $\Re{(a)}\leq 1$, $|z|<|a|$, let\\
\begin{equation*}
    F(z)=\prod_{n=1}^{\infty}(1-(\frac{z}{a_{n}})^2)
\end{equation*}
where $a_{n}=a+n-1$. Then by Lemma \ref{lemma4} we have $F(z)=\frac{\Gamma(a)^2}{\Gamma(a+z)\Gamma(a-z)}$. Now we let $G(z)=(zF(z))^{-1}$. Then the infinite partial fraction decomposition of $G$ is\\
\begin{equation*}
  G(z)=\frac{1}{z}+\sum_{n=1}^{\infty}\lambda_{n}(\frac{1}{z-a_{n}}+\frac{1}{z+a_{n}})
\end{equation*}
One can check\\
\begin{align*}
  F'(z) & =\Gamma(a)^2 (\frac{1}{\Gamma(a+z)\Gamma(a-z)}-\frac{z}{\Gamma(a+z)\Gamma(a-z)}(\psi(a+z)-\psi(a-z))) \\
\end{align*}
Hence
\begin{align*}
  \lambda_{n}=\frac{1}{a_{n}F'(a_{n})}=\frac{(-1)^n \Gamma(2a+n-1) }{\Gamma(a)^2 (a+n-1)(n-1)!}
\end{align*}
Let $k=n-1$, finally we obtain\\
\begin{equation}\label{gammap}
\Gamma(a+z)\Gamma(a-z)=\Gamma(a)^2+\sum_{k=0}^{\infty}\frac{(-1)^{k-1} \Gamma(2a+k) }{(a+k)k!}\frac{2z^2}{z^2-(a+k)^2}
\end{equation}
In order to give the convergence region of the series, just need to note that $|\Gamma(x+k)|<|\Gamma(\Re{(x)}+k)|$. If $\Re{(2a)}\leq1$, then\\
\begin{equation*}
    |\frac{(-1)^{k-1} \Gamma(2a+k) }{(a+k)k!}\frac{2z^2}{z^2-(a+k)^2}|\leq |\frac{2z^2}{a(z^2-a^2)}|
\end{equation*}
Therefore the series converges in $|z|<|a|$.\\
\end{proof}

\begin{corollary}
Suppose that $a\neq 0$, $0\leq Re(a)\leq 1$, then
\begin{equation*}
\zeta(2,a)=\sum_{k=0}^{\infty}\frac{2(-1)^{k} \Gamma(2a+k) }{\Gamma(a)^2 k! (a+k)^3}
\end{equation*}
where $\zeta(.,.)$ is Hurwitz zeta function.
\end{corollary}
\begin{proof}
It can be easily checked by applying Lemma \ref{lemma3}
\end{proof}

\begin{remark}
While $a=1$, the formula (\ref{dbgamma}) becomes\\
\begin{equation*}
    \Gamma(1+z)\Gamma(1-z)=1+\sum_{k=0}^{\infty}\frac{2(-1)^{k-1}z^2}{z^2-(1+k)^2}
\end{equation*}
This coincides with the partial fraction formula of $\pi z\csc(\pi z)$.\\
\end{remark}

\section{Acknowledgement}
Ich m\"{o}chte mich bei den Leuten bedanken, die im 2012/2013 mich online verleumdet hatten. Diese ungerechte Worte sind mir noch deutlich errinnerlish. Diese ungerechte Worte gaben mir Antrieb und machten mir ununterbrochen weiterkommen.

\bibliographystyle{unsrt}  


\newpage
\appendix
\section{Appendix}
In this appendix we provide some numerical comparison of those two series representations of $\zeta(n)$. We adopt the following notations. 
\begin{equation*}
    \zeta(m)=\sum_{n=1}^{\infty}A_{m}(n)=\sum_{n=1}^{\infty}(-1)^{n-1}B_{m}(n)
\end{equation*}
where
\begin{align*}
    &A_{m}(n)=\frac{1}{n^m}\\
    &B_{m}(n)=\frac{\rho_{m}(n)}{n^m}=\frac{m\Gamma(1-\omega_{m}n)...\Gamma(1-\omega_{m}^{m-1}n)}{n!n^m}\\
\end{align*}
And define\\
\begin{align*}
    &RA_{m}(n)=\zeta(m)-\sum_{k=1}^{n}\frac{1}{k^m}\\
    &RB_{m}(n)=\zeta(m)-\sum_{k=1}^{n}(-1)^{k-1}B_{m}(k)\\
\end{align*}
\begin{center}
    Table 1. Comparison for the terms and errors when $m=3$
\end{center}
\centering
\begin{tabular}{|c|c|c|c|c|}
\hline
$n$  & $A_{3}(n)$ & $B_{3}(n)$ & $RA_{3}(n)$ & $RB_{3}(n)$\\ \hline
$1$  & $1.000$  & $1.235$ & $2.021 \times 10^{-1}$ & $-3.343 \times 10^{-2}$ \\ \hline
$2$  & $1.250 \times 10^{-1}$  & $3.537 \times 10^{-2}$ & $7.706 \times 10^{-2}$ &  $1.939 \times 10^{-3}$\\ \hline
$3$  & $3.704 \times 10^{-2}$ & $2.091 \times 10^{-3}$ & $4.002 \times 10^{-2}$ & $-1.520 \times 10^{-4}$\\ \hline
$4$  & $1.562 \times 10^{-2}$ &  $1.660 \times 10^{-4}$ & $2.439 \times 10^{-2}$ & $1.405 \times 10^{-5}$\\ \hline
$5$  & $8.000 \times 10^{-3}$ &  $1.550 \times 10^{-5}$ & $1.639 \times 10^{-2}$ & $-1.442 \times 10^{-6}$\\ \hline
$6$  & $4.630 \times 10^{-3}$ &  $1.601 \times 10^{-6}$ & $1.177 \times 10^{-2}$ & $1.591 \times 10^{-7}$\\ \hline
$7$  & $2.915 \times 10^{-3}$ &  $1.776 \times 10^{-7}$ & $8.850 \times 10^{-3}$ & $-1.850 \times 10^{-8}$\\ \hline
$10$  & $1.000 \times 10^{-3}$ &  $3.155 \times 10^{-10}$ & $4.525 \times 10^{-3}$ & $3.584 \times 10^{-11}$\\ \hline
$20$  & $1.250 \times 10^{-4}$ &  $7.399 \times 10^{-19}$ & $1.189 \times 10^{-3}$ & $9.325 \times 10^{-20}$\\ \hline
$50$ & $8.000 \times 10^{-6}$ &  $1.748 \times 10^{-43}$ & $1.960 \times 10^{-4}$ & $2.348 \times 10^{-44}$\\ \hline
$100$ & $1.000 \times 10^{-6}$ &  $1.270 \times 10^{-83}$ & $4.950 \times 10^{-5}$ & $1.743 \times 10^{-84}$\\ \hline
\end{tabular}
\begin{center}
    Table 1. Comparison for the terms and errors when $m=5$
\end{center}
\centering
\begin{tabular}{|c|c|c|c|c|}
\hline
$n$  & $A_{5}(n)$ & $B_{5}(n)$ & $RA_{5}(n)$ & $RB_{5}(n)$\\ \hline
$1$  & $1.000$  & $1.038$ & $3.693 \times 10^{-2}$ & $-1.217 \times 10^{-3}$ \\ \hline
$2$  & $3.125 \times 10^{-1}$  & $1.221 \times 10^{-3}$ & $5.678 \times 10^{-3}$ &  $3.895 \times 10^{-6}$\\ \hline
$3$  & $4.115 \times 10^{-2}$ & $3.914 \times 10^{-6}$ & $1.563 \times 10^{-3}$ & $-1.883 \times 10^{-8}$\\ \hline
$4$  & $9.766 \times 10^{-3}$ &  $1.894 \times 10^{-8}$ & $5.860 \times 10^{-4}$ & $1.141 \times 10^{-10}$\\ \hline
$5$  & $3.200 \times 10^{-4}$ &  $1.149 \times 10^{-10}$ & $2.660 \times 10^{-4}$ & $-7.979 \times 10^{-13}$\\ \hline
$6$  & $1.286 \times 10^{-4}$ &  $8.041 \times 10^{-13}$ & $1.374 \times 10^{-4}$ & $6.159 \times 10^{-15}$\\ \hline
$7$  & $5.950 \times 10^{-4}$ &  $6.210 \times 10^{-15}$ & $7.787 \times 10^{-5}$ & $-5.110 \times 10^{-17}$\\ \hline
$10$  & $1.000 \times 10^{-5}$ &  $4.142 \times 10^{-21}$ & $2.041 \times 10^{-5}$ & $3.892 \times 10^{-23}$\\ \hline
$20$  & $3.125 \times 10^{-7}$ &  $6.089 \times 10^{-41}$ & $1.413 \times 10^{-6}$ & $6.724 \times 10^{-43}$\\ \hline
$50$ & $3.200 \times 10^{-9}$ &  $1.134 \times 10^{-98}$ & $3.843 \times 10^{-8}$ & $1.385 \times 10^{-100}$\\ \hline
$100$ & $1.000 \times 10^{-10}$ &  $1.276 \times 10^{-193}$ & $2.450 \times 10^{-9}$ & $1.612 \times 10^{-195}$\\ \hline
\end{tabular}

Xiaowei Wang(\begin{CJK}{UTF8}{gbsn}王骁威\end{CJK})\\
Institut f\"{u}r Mathematik, Universit\"at  Potsdam, Potsdam OT Golm, Germany\\
 Email: \texttt{xiawang@gmx.de}

\end{document}